\documentclass[12pt]{amsart}
\usepackage{amssymb}
\usepackage{amsmath}
\usepackage[active]{srcltx}
\usepackage{t1enc}
\usepackage[latin2]{inputenc}
\usepackage{verbatim}
\usepackage{amsmath,amsfonts,amssymb,amsthm}
\usepackage[mathcal]{eucal}
\usepackage{enumerate}
\usepackage[centertags]{amsmath}
\usepackage{graphics}

\newtheorem{theorem}{Theorem}

\newtheorem{lemma}{Lemma}

\newtheorem{corollary}{Corollary}

\newtheorem{Proposition}{Proposition}

\begin{document}
\author{Davit Baramidze, Nato Gogolashvili and Nato Nadirashvili}
\title[$T$ means]{Convergence of $T$ means with respect to Vilenkin systems of integrable functions}
\address{D. Baramidze, The University of Georgia, School of science and
	technology, 77a Merab Kostava St, Tbilisi 0128, Georgia and Department of
	Computer Science and Computational Engineering, UiT - The Arctic University
	of Norway, P.O. Box 385, N-8505, Narvik, Norway.}
\email{davit.baramidze@ug.edu.ge }
\address{N. Gogolashvili, The University of Georgia, School of science and
	technology, 77a Merab Kostava St, Tbilisi 0128, Georgia and Department of
	Computer Science and Computational Engineering, UiT - The Arctic University
	of Norway, P.O. Box 385, N-8505, Narvik, Norway.}
\email{n.gogolashvili@ug.edu.ge }
\address{N. Nadirashvili, The University of Georgia, School of science and
	technology, 77a Merab Kostava St, Tbilisi 0128, Georgia.}
\email{nato.nadirashvili@gmail.com  }
\thanks{The research was supported by Shota Rustaveli National Science
	Foundation grant no. PHDF-21-1702.}
%\thanks{The research was supported by Shota Rustaveli National Science Foundation grant no.FR-19-676.}
\date{}
\maketitle

\begin{abstract}
In this paper we derive converge of $T$ means of Vilenkin-Fourier
series with monotone coefficients of integrable functions in Lebesgue and
Vilinkin-Lebesgue points. Moreover, we discuss pointwise and norm convergence in $L_p$ norms of such $T$ means.
\end{abstract}



\section{Introduction}

The definitions and notations used in this introduction can be
found in our next Section. 

It is well-known (see e.g. the book \cite{sws}) that there exists an
absolute constant $c_{p},$ depending only on $p,$ such that 
\begin{equation*}
\left\Vert S_{n}f\right\Vert _{p}\leq c_{p}\left\Vert f\right\Vert _{p},%
\text{ \ when \ }p>1.
\end{equation*}%
On the other hand, (for details see \cite{BPT,BPT1,PTW3,tep7,tep9,tut1}) boundedness does not hold for $p=1.$ The analogue of Carleson's theorem for Walsh system was proved by Billard 
\cite{Billard1967} for $p=2$ and by Sjölin \cite{sj1} for $1 <p<\infty$,
while for bounded Vilenkin systems by Gosselin \cite{goles}. For
Walsh-Fourier series, Schipp \cite{s1,s2,sws} gave a proof by using methods
of martingale theory. A similar proof for Vilenkin-Fourier series can be
found in Schipp and Weisz \cite{s3,wk}. In each proof, they show that the
maximal operator of the partial sums is bounded on $L_p$, i.e. there exists
an absolute constant $c_p$ such that 
\begin{equation*}
\left\Vert S^{\ast }f\right\Vert_p\leq c_p\left\Vert f\right\Vert_p,\text{ \
	when \ }f\in L_p,\text{ \ } p>1.
\end{equation*}
Hence, if $f\in L_p(G_m),$ where $p>1,$ then 
$
S_{n}f\to f, \ \ \text{a.e. on } \ \ G_m.
$
Stein \cite{steindiv} constructed the integrable function whose
Vilenkin-Fourier (Walsh-Fourier) series diverges almost everywhere. In \cite%
{sws} was proved that there exists an integrable function whose
Walsh-Fourier series diverges everywhere. a.e convergence of subsequences of Vilenkin-Fourier series was considered in 
\cite{bnpt}, where was used methods of martingale Hardy spaces. 

If we consider the following restricted maximal operator $\widetilde{S}%
_{\#}^{\ast}f:=\sup_{n\in\mathbb{N}}\left\vert S_{M_n}f\right\vert,$  we
have weak $(1,1)$ type inequality for $f\in L_1(G_m).$  
Hence, if $f\in L_1(G_m),$ then $S_{M_n}f\to f, \ \text{a.e. on } \  G_m.$
Moreover, for any integrable function it is known that a.e. point is
Lebesgue point and for any such point $x$ of integrable function $f$ we
have that 
\begin{equation}  \label{smnvl}
S_{M_n}f(x)\to f(x), \ \text{as} \ n\to\infty, \ \text{for any Lebesgue point} \ x
\ \text{of } \ f\in L_1(G_m).
\end{equation}

In the one-dimensional case Yano \cite{Yano} proved that 
\begin{equation*}
\left\Vert \sigma _{n}f-f\right\Vert _{p}\rightarrow 0,\text{ \ \ \ when \ \
	\ \ }n\rightarrow \infty ,\text{ \ }(f\in L_p(G_m),\text{ \ }1\leq p\leq
\infty ).
\end{equation*}

If we consider the maximal operator of F\'ejer means
\begin{eqnarray*}
\sigma^{\ast}f:=\sup_{n\in\mathbb{N}}\left\vert \sigma_{n}f\right\vert, 
\end{eqnarray*}
then 
\begin{eqnarray*}
	\lambda \mu\left\{ \sigma ^{\ast}f>\lambda \right\} \leq c\left\Vert
	f\right\Vert_{1}, \ \ \ f\in L_1(G_m), \ \ \lambda>0.
\end{eqnarray*}
This result can be found in Zygmund \cite{Zy} for the trigonometric series,
in Schipp \cite{Sc}  and \cite{GNT,pt,PTT,tep1,tep10,tep11}) for Walsh series and in P\'al, Simon \cite{PS} for bounded Vilenkin series (see also Weisz \cite{We1,We2}). The boundedness does not hold from Lebesgue space $L_1(G_m)$ to the space $
L_1(G_m)$. The $\text{weak}-(1,1)$ type inequality  follows that for any $f\in
L_1(G_m),$ 
\begin{equation*}
\sigma_nf(x)\to f(x), \ \ \ \text{a.e., as} \ \ \ n\to\infty.
\end{equation*}
Moreover, in \cite{goggog}  (see also \cite{Ga2}) was proved that for any integrable function it is
known that a.e. point is Vilenkin-Lebesgue points and for any such point $x$
of integrable function $f$ we have that $$\sigma_nf(x)\to f(x), \ \ \text{as}
\ \ n\to\infty.$$

M\'oricz and Siddiqi \cite{Mor} investigate  approximation properties of
some special N\"orlund means of Walsh-Fourier series of $L_{p}$ functions in
norm. Similar results for the two-dimensional case can be found in Nagy \cite%
{nagy,n}, Nagy and Tephnadze \cite{NT1,NT2,NT3,NT4}, Gogolashvili and Tephnadze \cite{GT1,GT2} (see also \cite{BPTW}, \cite{MPT}). Approximation properties of general summability methods can be found
in \cite{BN,BNT}. Fridli, Manchanda and Siddiqi \cite{FMS} improved and
extended results of M\'oricz and Siddiqi \cite{Mor} to Martingale Hardy
spaces. The a.e. convergence of N\"orlund means of Vilenkin-Fourier series with
monotone coefficients of $f\in L_1$ was proved in \cite{PTW} (see also \cite{PTW2}).
In \cite{tut3} was proved that the maximal operators of $T$ means  $T^{\ast }$ defined by
$T^{\ast }f:=\sup_{n\in \mathbb{N}}\left\vert T_{n}f\right\vert$
either with non-increasing coefficients, or non-decreasing sequence satisfying condition 
\begin{equation}
\frac{q_{n-1}}{Q_{n}}=O\left( \frac{1}{n}\right) ,\text{ \ \ as \ \ }\
n\rightarrow \infty   \label{fn01}
\end{equation}
are bounded from the Hardy space $H_{1/2}$ to the space $weak-L_{1/2}$. Moreover, there exists a martingale and such $T $ means for which boundedness does not hold from the Hardy space $H_{p}$ to the space $L_{p}$  when $0<p\leq 1/2.$

One of the most well-known mean of $T$ means is the Riesz summability. In \cite{tep6} (see also \cite{LPTT}) it was proved that the maximal operator of Riesz logarithmic means 
\begin{eqnarray*}
R^{\ast}f:=\sup_{n\in\mathbb{N}}\left\vert R_{n}f\right\vert
\end{eqnarray*}
is bounded from the Hardy space $H_{1/2}$ to the space $weak-L_{1/2}$ and is not bounded from $H_{p}$ to the space $L_{p},$ for $0<p\leq 1/2.$ There was also proved that Riesz summability has better properties than Fejér means. 

In this paper we derive convergence of $T$ means of Vilenkin-Fourier
series with monotone coefficients of integrable functions in Lebesgue and
Vilinkin-Lebesgue points.

This paper is organized as follows: In order not to disturb our discussions
later on some definitions and notations are presented in Section 2. For the
proofs of the main results we need some auxiliary Lemmas, some of them are
new and of independent interest. These results are presented in Section 3.
The main results and some of its consequences and detailed proofs  are given in Section 4.

\section{Definitions and Notation}

Denote by $\mathbb{N}_{+}$ the set of the positive integers, $\mathbb{N}:=\mathbb{N}_{+}\cup \{0\}.$ Let $m:=(m_{0,}$ $m_{1},...)$ be a sequence of the positive
integers not less than 2. Denote by 
\begin{equation*}
Z_{m_{k}}:=\{0,1,...,m_{k}-1\}
\end{equation*}
the additive group of integers modulo $m_{k}$.

Define the Vilenkin group $G_{m}$ as the complete direct product of the groups $
Z_{m_{i}}$ with the product of the discrete topologies of $Z_{m_{j}}`$s (for details see \cite{Vi}). T In this paper we discuss bounded Vilenkin groups, i.e. the case
when $\sup_{n}m_{n}<\infty .$ The direct product $\mu $ of the measures 
$
\mu _{k}\left( \{j\}\right) :=1/m_{k}\text{ \  }(j\in Z_{m_{k}})
$
is the Haar measure on $G_m$ with $\mu \left( G_{m}\right) =1.$
The elements of $G_{m}$ are represented by sequences 
\begin{equation*}
x:=\left( x_{0},x_{1},...,x_{j},...\right) ,\ \left( x_{j}\in
Z_{m_{j}}\right).
\end{equation*}

It is easy to give a basis for the neighborhoods of $G_{m}:$ 
\begin{equation*}
I_{0}\left( x\right) :=G_{m},\text{ \ }I_{n}(x):=\{y\in G_{m}\mid
y_{0}=x_{0},...,y_{n-1}=x_{n-1}\}, \ \text{ where } \ x\in G_{m}, \ n\in \mathbb{N}.
\end{equation*}

If we define the so-called generalized number system based on $m$
in the following way: 
\begin{equation*}
M_{0}:=1,\ M_{k+1}:=m_{k}M_{k}\,\,\,\ \ (k\in \mathbb{N}),
\end{equation*}%
then every $n\in \mathbb{N}$ can be uniquely expressed as $n=\sum_{j=0}^{\infty }n_{j}M_{j},$ where 
$n_{j}\in Z_{m_{j}}$ $(j\in \mathbb{N}_{+})$ and only a finite number of $n_{j}`$s differ from zero.

We introduce on $G_{m}$ an orthonormal system which is called the Vilenkin
system. First, we define  complex-valued function 
$r_{k}\left(x\right) :G_{m}\rightarrow \mathbb{C},$ the generalized Rademacher functions, by
\begin{equation*}
r_{k}\left( x\right) :=\exp \left( 2\pi ix_{k}/m_{k}\right) ,\text{ }\left(
i^{2}=-1,x\in G_{m},\text{ }k\in \mathbb{N}\right) .
\end{equation*}

Next, we define the Vilenkin system 
$\,\,\,\psi :=(\psi _{n}:n\in \mathbb{N})$ on $G_{m}$ by: 
\begin{equation*}
\psi _{n}(x):=\prod\limits_{k=0}^{\infty }r_{k}^{n_{k}}\left( x\right)
,\,\,\ \ \,\left( n\in \mathbb{N}\right) .
\end{equation*}

Specifically, we call this system the Walsh-Paley system when $m\equiv 2.$

The norms (or quasi-norms) of the spaces $L_{p}(G_{m})$ and $%
weak-L_{p}\left( G_{m}\right) $ $\left( 0<p<\infty \right) $ are
respectively defined by 
\begin{equation*}
\left\Vert f\right\Vert _{p}^{p}:=\int_{G_{m}}\left\vert f\right\vert
^{p}d\mu ,\text{ }\left\Vert f\right\Vert _{weak-L_{p}}^{p}:=\underset{%
\lambda >0}{\sup }\lambda ^{p}\mu \left( f>\lambda \right) <+\infty .
\end{equation*}%

The Vilenkin system is orthonormal and complete in $L_{2}\left( G_{m}\right) 
$ (see \cite{Vi}).

Now, we introduce analogues of the usual definitions in Fourier-analysis. If 
$f\in L_{1}\left( G_{m}\right) $ we can define Fourier coefficients, partial
sums and Dirichlet kernels with respect to the Vilenkin system in the usual manner: 
\begin{equation*}
\widehat{f}\left( n\right) :=\int_{G_{m}}f\overline{\psi }_{n}d\mu,\ \ \ \
\ \ \
S_{n}f:=\sum_{k=0}^{n-1}\widehat{f}\left( k\right) \psi _{k},\text{ \ \ }%
D_{n}:=\sum_{k=0}^{n-1}\psi _{k\text{ }},\text{ \ \ }\left( n\in 
\mathbb{N}_{+}\right).
\end{equation*}

%It is well known that if $n\in\mathbb{N},$ then
%\begin{equation} \label{8dn}
%D_{M_n}\left(x\right)=\left\{ \begin{array}{ll} M_n, & x\in I_n, \\
%0, & x\notin I_n. \end{array} \right.
%\end{equation}
%Moreover, if  $n=\sum_{i=0}^{\infty}n_iM_i,$ and $1\leq s_n\leq m_n-1,$ then we have the following identity: 
%\begin{equation} \label{9dn}
%D_n=\psi_n\left(\sum_{j=0}^{\infty}D_{M_j}\sum_{k=m_j-n_j}^{m_j-1}r_j^k\right),
%\end{equation}

Recall that 
%\begin{equation*}
%D_{M_{n}}\left( x\right) =\left\{
%\begin{array}{ll}
%M_{n}, & \text{if\thinspace \thinspace \thinspace }x\in I_{n}, \\
%0, & \text{if}\,\,x\notin I_{n}.%
%\end{array}%
%\right.  \label{1dn}
%\end{equation*}
\begin{eqnarray}  \label{dn22}
&&\int_{G_{m}}D_n(x)dx=1, \\
&&D_{M_n-j}(x) =D_{M_n}(x)-\psi_{M_n-1}(x)\overline{D}_j(x), \ j<M_n. \label{dn23}
\end{eqnarray}

The convolution of two functions $f,g\in L_{1}(G_m)$ is defined by 
\begin{equation*}
\left( f\ast g\right) \left( x\right) :=\int_{G_m}f\left( x-t\right) g\left(
t\right) dt\text{ \ \ }\left( x\in G_m\right).
\end{equation*}%
It is easy to see that if $f\in L_{p}\left( G_m\right) ,$ $g\in  L_{1}\left(
G_m\right) $ and $1\leq p<\infty .$ Then $f\ast g\in  L_{p}\left( G_m\right) 
$ and 
\begin{equation}  \label{concond}
\left\Vert f\ast g\right\Vert_{p}\leq \left\Vert f\right\Vert_{p}\left\Vert
g\right\Vert _{1}.
\end{equation}

Let $\{q_{k}:k\geq 0\}$ be a sequence of nonnegative numbers. The $n$-th Nörlund mean $t_{n}$ for a Fourier series of $f$ \ is defined by 
\begin{equation} \label{nor0}
t_{n}f=\frac{1}{Q_{n}}\overset{n}{\underset{k=1}{\sum }}q_{n-k}S_{k}f, \ \ \ \text{where} \ \ \ Q_{n}:=\sum_{k=0}^{n-1}q_{k}.
\end{equation}%

It is obvious that  \  \
$$
t_nf\left(x\right)=\underset{G_m}{\int}f\left(t\right)F_n\left(x-t\right) d\mu\left(t\right),
\ \ \
\text{where} \ \  \	F_n:=\frac{1}{Q_n}\overset{n-1}{\underset{k=0}{\sum }}q_{k}D_k$$ 
 is called  $T$ kernel.

\begin{Proposition} \label{corollary3n}
	Let $\{q_k:k\in \mathbb{N}\}$ be a sequence of non-increasing numbers. Then, for any $n,N\in \mathbb{N_+}$,
	\begin{eqnarray} \label{1.71alpha}
	&&\int_{G_m} F_{M_n}(x) d\mu (x)=1, \\
	&&\sup_{n\in\mathbb{N}}\int_{G_m}\left\vert F_{M_n}(x)\right\vert d\mu(x)\leq c<\infty,\label{1.72alpha} \\
	&&\sup_{n\in\mathbb{N}}\int_{G_m \backslash I_N}\left\vert F_{M_n}(x)\right\vert d\mu (x)\rightarrow  0, \ \ \text{as} \ \ n\rightarrow  \infty, \label{1.73alpha}
	\end{eqnarray}
\end{Proposition}

Let $\{q_{k}:k\geq 0\}$ be a sequence of non-negative numbers. The $n$-th  $T$ means $T_n$ for a Fourier series of $f$ are defined by
\begin{equation} \label{nor}
T_nf:=\frac{1}{Q_n}\overset{n-1}{\underset{k=0}{\sum }}q_{k}S_kf, \ \ \ \text{where} \ \ \  Q_{n}:=\sum_{k=0}^{n-1}q_{k}.
\end{equation}

It is obvious that  \  \
$$
T_nf\left(x\right)=\underset{G_m}{\int}f\left(t\right)F^{-1}_n\left(x-t\right) d\mu\left(t\right),  \ \ \text{
where} \ \  	F^{-1}_n:=\frac{1}{Q_n}\overset{n-1}{\underset{k=0}{\sum }}q_{k}D_k$$ 
is called the $T$ kernel.

We always assume that $\{q_k:k\geq 0\}$ is a sequence of non-negative numbers and $q_0>0.$ Then the summability method (\ref{nor}) generated by $\{q_k:k\geq 0\}$ is regular if and only if 
$$\lim_{n\rightarrow\infty}Q_n=\infty.$$

It is easy to show that, for any real numbers  $a_1,\dots, a_m,$ $b_1,\dots, b_m$ and  $a_k=A_k-A_{k-1},$ $k=n, \dots,m,$  we have so called Abel transformation:
\begin{eqnarray*}
	\sum_{k=m} ^{n}a_k b_k 
	&=& A_n b_n -A_{m-1} b_m + \sum_{k=m} ^{n-1}A_{k} (b_k -b_{k+1}).
\end{eqnarray*}

For $a_j=A_j-A_{j-1}, \ j=1,...,n,$ if we invoke Abel transformations 
\begin{eqnarray}\label{abel1} \overset{n-1}{\underset{j=1}{\sum}}a_jb_j&=&A_{n-1}b_{n-1}-A_{0}b_1+\overset{n-2}{\underset{j=0}{\sum}}A_j(b_j-b_{j+1}),  \\ \label{abel2}
\overset{n-1}{\underset{j=M_N}{\sum}}a_jb_j&=&A_{n-1}b_{n-1}-A_{M_N-1}b_{M_N}+\overset{n-2}{\underset{j=M_N}{\sum}}A_j(b_j-b_{j+1}),
\end{eqnarray} 

For $b_j=q_j$, $a_j=1$ and $A_j=j$ for any $j=0,1,...,n$ we get the following identity:
\begin{eqnarray} \label{2b}
Q_n&=&\overset{n-1}{\underset{j=0}{\sum}}q_j=q_0+\overset{n-1}{\underset{j=1}{\sum}}q_j =q_0+\overset{n-2}{\underset{j=1}{\sum}}\left(q_{j}-q_{j+1}\right) j+q_{n-1}{(n-1)}, \\
\label{2b0}
\overset{n-1}{\underset{j=M_N}{\sum}}q_j &=&\overset{n-2}{\underset{j=M_N}{\sum}}\left(q_{j}-q_{j+1}\right) j+q_{n-1}{(n-1)}-(M_N-1)q_{M_N}.
\end{eqnarray}
Moreover, if  use  $D_0=K_0=0 \ \text{for any}  \ x\in G_m $
and invoke Abel transformations \eqref{abel1} and \eqref{abel2} for $b_j=q_j$, $a_j=D_j$ and $A_j=jK_j$ for any $j=0,1,...,n-1$ we get identities: 
\begin{eqnarray} 	\label{2cTmean}
&& \ \ \ F_n^{-1}=\frac{1}{Q_n}\overset{n-1}{\underset{j=0}{\sum }}q_{j}D_j=\frac{1}{Q_n}\left(\overset{n-2}{\underset{j=1}{\sum}}\left(q_j-q_{j+1}\right) jK_j+q_{n-1}(n-1)K_{n-1}\right),\\
\label{2c0Tmean} 
&&\frac{1}{Q_n}\overset{n-1}{\underset{j=M_N}{\sum }}q_jD_j\\ \notag
&=&\frac{1}{Q_n}\left(\overset{n-2}{\underset{j=M_N}{\sum}}\left(q_j-q_{j+1}\right) jK_j+q_{n-1}(n-1)K_{n-1}-q_{M_N}{(M_N-1)} K_{{M_N-1}}\right).
\end{eqnarray}
Analogously, if  use  $S_0f=\sigma_0f=0, \  \text{for any} \  x\in G_m $  and invoke Abel transformations \eqref{abel1} and \eqref{abel2} for $b_j=q_j$, $a_j=S_j$ and $A_j=j\sigma_j$ for any $j=0,1,...,n-1$ we get identities: 
\begin{eqnarray} \label{2cc2}
&&T_nf=\frac{1}{Q_n}\overset{n-1}{\underset{j=0}{\sum }}q_{j}S_jf=\frac{1}{Q_n}\left(\overset{n-2}{\underset{j=1}{\sum}}\left(q_j-q_{j+1}\right) j\sigma_jf+q_{n-1}(n-1)\sigma_{n-1}f\right),
\\	\label{2cc20} 
&&\frac{1}{Q_n}\overset{n-1}{\underset{j=M_N}{\sum }}q_jS_jf \\ \notag
&=&\frac{1}{Q_n}\left(\overset{n-2}{\underset{j=M_N}{\sum}}\left(q_j-q_{j+1}\right) j\sigma_jf+q_{n-1}(n-1)\sigma_{n-1}f-q_{M_N}{(M_N-1)} \sigma_{{M_N-1}}f\right).
\end{eqnarray}

If $q_k\equiv 1$ in \eqref{nor0} and \eqref{nor}  we respectively define the Fejér means $\sigma _{n}$ and Fejér Kernels $K_{n}$ as follows: 
\begin{equation*}
\sigma _{n}f:=\frac{1}{n}\sum_{k=1}^{n}S_{k}f\,,\text{ \ \ }K_{n}:=\frac{1}{n%
}\sum_{k=1}^{n}D_{k}.
\end{equation*}

It is well-known that (for details see \cite{AVD})
\begin{equation} \label{fn5}
n\left\vert K_{n}\right\vert \leq c\sum_{l=0}^{\left\vert n\right\vert
}M_{l}\left\vert K_{M_{l}}\right\vert  
\end{equation} 
and for any $n,N\in \mathbb{N_+}$,
\begin{eqnarray} \label{fn40}
&&\int_{G_m} K_n (x)d\mu(x)=1,\\
&& \label{fn4}
\sup_{n\in\mathbb{N}}\int_{G_m}\left\vert K_n(x)\right\vert d\mu(x)\leq c<\infty,\\
&& \label{fn400}
\sup_{n\in\mathbb{N}}\int_{G_m \backslash I_N}\left\vert K_n(x)\right\vert d\mu (x)\rightarrow  0, \ \ \text{as} \ \ n\rightarrow  \infty, 
\end{eqnarray}

The well-known example of N\"orlund summability is the so-called $\left(C,\alpha\right)$ mean (Ces\`aro means) for $0<\alpha<1,$ which are defined by
\begin{equation*}
\sigma_n^{\alpha}f:=\frac{1}{A_n^{\alpha}}\overset{n}{\underset{k=1}{
		\sum}}A_{n-k}^{\alpha-1}S_kf, \ \ \
\text{where } \ \ \
A_0^{\alpha}:=0,\qquad A_n^{\alpha}:=\frac{\left(\alpha+1\right)...\left(\alpha+n\right)}{n!}.
\end{equation*}

We also consider the "inverse" $\left(C,\alpha\right)$ means, which is an example of $T$ means:
\begin{equation*}
U_n^{\alpha}f:=\frac{1}{A_n^{\alpha}}\overset{n-1}{\underset{k=0}{\sum}}A_{k}^{\alpha-1}S_kf, \qquad 0<\alpha<1.
\end{equation*}

Let $V_n^{\alpha}$ denote
the $T$ mean, where $	\left\{q_0=0, \  q_k=k^{\alpha-1}:k\in \mathbb{N}_+\right\} ,$
that is 
\begin{equation*}
V_n^{\alpha}f:=\frac{1}{Q_n}\overset{n-1}{\underset{k=1}{\sum }}k^{\alpha-1}S_kf,\qquad 0<\alpha<1.
\end{equation*}

The $n$-th Riesz logarithmic mean $R_{n}$ and the Nörlund logarithmic mean
$L_{n}$ are defined by
\begin{equation*}
R_{n}f:=\frac{1}{l_{n}}\sum_{k=1}^{n-1}\frac{S_{k}f}{k}\text{ \ \ \ and  \ \ \ }
L_{n}f:=\frac{1}{l_{n}}\sum_{k=1}^{n-1}\frac{S_{k}f}{n-k}, \ \  \
\text{ where } \ \ \ l_{n}:=\sum_{k=1}^{n-1}1/k.
\end{equation*}

Up to now we have considered $T$ means in the case when the sequence $\{q_k:k\in\mathbb{N}\}$ is bounded but now we consider $T$ summabilities with unbounded sequence $\{q_k:k\in\mathbb{N}\}$. 

Let $\alpha\in
\mathbb{R}_+,\ \ \beta\in\mathbb{N}_+$ and
$
\log^{(\beta)}x:=\overset{\beta-\text{times}}{\overbrace{\log ...\log}}x.
$
If we define the sequence $\{q_k:k\in \mathbb{N}\}$ by
$	\left\{q_0=0, \ q_k=\log^{\left(\beta \right)}k^{\alpha
}:k\in\mathbb{N}_+\right\},$ 
then we get the class $B_n^{\alpha,\beta}$ of $T$ means with non-decreasing coefficients:
\begin{equation*}
B_n^{\alpha,\beta}f:=\frac{1}{Q_n}
\sum_{k=1}^{n-1}\log^{\left(\beta\right)}k^{\alpha}S_kf.
\end{equation*}

%It is obvious that $\frac{n}{2}\log^{\left(\beta \right)}\frac{n^{\alpha }}{2^{\alpha }}\leq Q_n\leq n\log^{\left(\beta\right)}n^{\alpha}.$ It follows that
%\begin{eqnarray} \label{node00}
%\frac{q_{n-1}}{Q_n}\leq\frac{c\log^{\left(\beta\right)}n^{\alpha}}{n\log^{\left(\beta\right) }n^{\alpha}}= O\left(\frac{1}{n}\right)\rightarrow 0,\text{ \ as \ }n\rightarrow \infty.
%\end{eqnarray}

\section{Auxiliary lemmas}

First we consider kernels of $T$ kernels with non-increasing sequences:

\begin{lemma}\label{lemma0nn}
Let $\{q_k:k\in\mathbb{N}\}$ be a sequence of non-increasing numbers, satisfying the condition
\begin{equation} \label{fn011}
\frac{q_{0}}{Q_n}=O\left( \frac{1}{n}\right) ,\text{ \ \ as \ \ } n\rightarrow\infty.
\end{equation}
Then, for some constant $c,$ we have that
\begin{equation*}
\left\vert F^{-1}_n\right\vert\leq\frac{c}{n}\left\{\sum_{j=0}^{\left\vert n\right\vert }M_j\left\vert K_{M_j}\right\vert \right\}.
\end{equation*}
\end{lemma}
\begin{proof}
Let  sequence $\{q_k:k\in \mathbb{N}\}$ be non-increasing. Then, by using (\ref{fn01}) we get that
	\begin{eqnarray*}
		\frac{1}{Q_n}\left(\overset{n-2}{\underset{j=1}{\sum }}\left\vert
		q_{j}-q_{j+1}\right\vert+q_{n-1}\right) 
		\leq\frac{1}{Q_n}\left(\overset{n-2}{\underset{j=1}{\sum }}\left(
		q_{j}-q_{j+1}\right)+q_{n-1}\right) 
		\leq \frac{q_{0}}{Q_{n}}\leq \frac{c}{n}.
	\end{eqnarray*}
	
	Hence, if we apply \eqref{fn5} and use the equalities (\ref{2b}) and (\ref{2cTmean}) we immediately get that
\begin{eqnarray*}
	\left\vert F^{-1}_n\right\vert \leq  \left( \frac{1}{Q_n}\left( \overset{n-1}{\underset{j=1}{\sum }}\left\vert q_{j}-q_{j+1} \right\vert+q_{n-1}\right)\right)\sum_{i=0}^{\left\vert n\right\vert } M_i\left\vert K_{M_i}\right\vert  \leq \frac{c}{n}\sum_{i=0}^{\left\vert n\right\vert }M_i\left\vert K_{M_i}\right\vert.
\end{eqnarray*}
	The proof is complete by just combining the estimates above.
\end{proof}
\begin{lemma}\label{lemma0nnT11}
	Let $\{q_k:k\in\mathbb{N}\}$ be a sequence of non-increasing numbers. Then, for any $n,N\in \mathbb{N_+}$,
	\begin{eqnarray} \label{1.71alphaT1}
	&&\int_{G_m} F^{-1}_n(x) d\mu (x)=1, \\
	&&\sup_{n\in\mathbb{N}}\int_{G_m}\left\vert F^{-1}_n(x)\right\vert d\mu(x)<\infty,\label{1.72alphaT1} \\
	&&\sup_{n\in\mathbb{N}}\int_{G_m \backslash I_N}\left\vert F^{-1}_n(x)\right\vert d\mu (x)\rightarrow  0, \ \ \text{as} \ \ n\rightarrow  \infty. \label{1.73alphaT1}
	\end{eqnarray}
\end{lemma}
\begin{proof}
According to \eqref{dn22} we easily obtain proof of \eqref{1.71alphaT1}. By using \eqref{fn4} combined with (\ref{2b}) and (\ref{2cTmean}) we get that
\begin{eqnarray*}
&&\frac{1}{Q_n}\left(\overset{n-2}{\underset{j=0}{\sum}}\left(q_j-q_{j+1}\right)j\int_{G_m}\left\vert K_j\right\vert d\mu +q_{n-1}(n-1)\int_{G_m}\vert K_{n-1}\vert d\mu \right)\\
&\leq&\frac{c}{Q_n}\left(\overset{n-2}{\underset{j=0}{\sum}}\left(q_j-q_{j+1}\right)j+q_{n-1}(n-1)\right)\leq c<\infty,
\end{eqnarray*}
so also \eqref{1.72alphaT1} is proved. 
By using \eqref{fn400}  and inequalities (\ref{2b}) and (\ref{2cTmean}) we can conclude that

\begin{eqnarray*}
\int_{G_m \backslash I_N}\left\vert F^{-1}_n\right\vert d\mu 
&\leq&\frac{1}{Q_n}\overset{n-1}{\underset{j=0}{\sum}}\left(q_j-q_{j+1}\right)j\int_{G_m \backslash I_N}\left\vert K_j\right\vert d\mu  +\frac{q_{n-1}(n-1)}{Q_n}\int_{G_m \backslash I_N}\vert K_{n-1}\vert\\
&\leq&\frac{1}{Q_n}\overset{n-2}{\underset{j=0}{\sum}}\left(q_j-q_{j+1}\right)j\alpha_j+\frac{q_{n-1}(n-1)\alpha_{n-1}}{Q_n}=I+II,\\ \notag
\end{eqnarray*}
where  $\alpha_n\to 0, \ \ \text{as} \ \ n\to\infty.$ Since sequence is non-increasing we can conclude that
$$II=\frac{q_{n-1}(n-1)\alpha_{n-1}}{Q_n}\leq \alpha_{n-1}\to 0, \ \ \text{as} \ \ n\to\infty.$$
	
Moreover, for any $\varepsilon>0$ there exists $N_0\in \mathbb{N},$ such that $\alpha_n< \varepsilon$ when $n>N_0.$ Furthermore,
	\begin{eqnarray*}
		I=\frac{1}{Q_n}\overset{n-2}{\underset{j=0}{\sum}}\left(q_j-q_{j+1}\right)j\alpha_j
		=\frac{1}{Q_n}\overset{N_0}{\underset{j=0}{\sum}}\left(q_j-q_{j+1}\right)j\alpha_j
		+\frac{1}{Q_n}\overset{n-2}{\underset{j=N_0+1}{\sum}}\left(q_j-q_{j+1}\right)j\alpha_j:=I_1+I_2.
	\end{eqnarray*}
	The sequence $\{q_k:k\in\mathbb{N}\}$ is non-increasing and therefore $\vert q_j-q_{j+1}\vert<2q_0,$
	$$I_1\leq \frac{2q_0N_0}{Q_n}\to 0, \ \ \ \text{as} \ \ \ n\to \infty$$
	and
	\begin{eqnarray*}
		I_2=\frac{1}{Q_n}\overset{n-2}{\underset{j=N_0+1}{\sum}}\left(q_j-q_{j+1}\right)j\alpha_j 
		\leq \frac{\varepsilon}{Q_n}\overset{n-2}{\underset{j=N_0+1}{\sum}}\left(q_j-q_{j+1}\right)j
		\leq \frac{\varepsilon}{Q_n}\overset{n-2}{\underset{j=0}{\sum}}\left(q_j-q_{j+1}\right)j<\varepsilon.
	\end{eqnarray*}
	and we can conclude that $I_2\to 0$ so proof is complete.
\end{proof}

Next we consider kernels of $T$ means with non-decreasing sequences:
\begin{lemma}\label{lemma0nnT1}
	Let $\{q_k:k\in\mathbb{N}\}$ be a sequence of non-decreasing numbers, satisfying the condition
	\begin{equation} \label{fn01T1}
	\frac{q_{n-1}}{Q_n}=O\left( \frac{1}{n}\right) ,\text{ \ \ as \ \ }
	n\rightarrow\infty.
	\end{equation}
	Then for some constant $c,$
	\begin{equation*}
	\left\vert F^{-1}_n\right\vert\leq\frac{c}{n}\left\{\sum_{j=0}^{\left\vert n\right\vert }M_j\left\vert K_{M_j}\right\vert \right\}.
	\end{equation*}
\end{lemma}
\begin{proof}
	Since the sequence $\{q_k:k\in \mathbb{N}\}$ be non-decreasing if we apply the condition \eqref{fn01T1} we find that
\begin{eqnarray*} \label{non-increasing-T1}
&&\frac{1}{Q_n}\left(\overset{n-2}{\underset{j=1}{\sum}}\left\vert q_j-q_{j+1}\right\vert+q_{n-1}\right)  =\frac{1}{Q_n}\left(\overset{n-2}{\underset{j=1}{\sum }}\left (q_{j+1}-q_j \right)+q_{n-1}\right)\leq \frac{2q_{n-1}}{Q_{n}}\leq \frac{c}{n}.
\end{eqnarray*}
If we apply Abel transformation \eqref{2c0Tmean} combined with \eqref{fn5}  and \eqref{non-increasing-T1} we get that
	\begin{eqnarray*}
		\left\vert F^{-1}_n\right\vert &\leq & \left( \frac{1}{Q_n}\left( \overset{n-1}{\underset{j=1}{\sum }}\left\vert q_{j}-q_{j+1} \right\vert+q_{n-1}+q_0\right)\right)\sum_{i=0}^{\left\vert n\right\vert } M_i\left\vert K_{M_i}\right\vert 
		\leq  \frac{c}{n}\sum_{i=0}^{\left\vert n\right\vert }M_i\left\vert K_{M_i}\right\vert.
	\end{eqnarray*}
	The proof is complete.
\end{proof}

\begin{lemma}\label{lemma0nnT1A}
	Let $\{q_k:k\in\mathbb{N}\}$ be a sequence of non-decreasing numbers, satisfying the condition \eqref{fn01T1}.
	Then, for some constant $c,$
	\begin{eqnarray} \label{1.71alphaT1A}
	&&\int_{G_m} F^{-1}_n(x) d\mu (x)=1, \\
	&&\sup_{n\in\mathbb{N}}\int_{G_m}\left\vert F^{-1}_n(x)\right\vert d\mu(x)\leq c<\infty,\label{1.72alphaT1A} \\
	&&\sup_{n\in\mathbb{N}}\int_{G_m \backslash I_N}\left\vert F^{-1}_n(x)\right\vert d\mu (x)\rightarrow  0, \ \ \text{as} \ \ n\rightarrow  \infty. \label{1.73alphaT1A}
	\end{eqnarray}
\end{lemma}
\begin{proof}
		If we compare the estimation of $F_n$ in  Lemma \ref{lemma0nnT11} with the estimation of $F_n$ in Lemma \ref{lemma0nnT1}   we find that they are quite same. It follows that proof is analogical to Lemma \ref{lemma0nnT11}. So, we leave out the details.
\end{proof}

Finally, we study some special subsequences of kernels of  $T$ means:

\begin{lemma}\label{lemma0nnT121}Let $n\in \mathbb{N}$. Then
	\begin{eqnarray} \label{1.71alphaT2j1} F^{-1}_{M_n}(x)=D_{M_n}(x)-\psi_{M_n-1}(x)\overline{F}_{M_n}(x).
	\end{eqnarray}
\end{lemma}
\begin{proof} By using \eqref{dn23} we get that
	\begin{eqnarray*}
		F^{-1}_{M_n}(x)&=&\frac{1}{Q_{M_n}}\overset{M_n-1}{\underset{k=0}{\sum}}q_{k}D_k(x)=\frac{1}{Q_{M_n}}\overset{M_n}{\underset{k=1}{\sum }}q_{M_n-k}D_{M_n-k}(x)\\
		&=&\frac{1}{Q_{M_n}}\overset{M_n}{\underset{k=1}{\sum }}q_{M_n-k}\left(D_{M_n}(x)-\psi_{M_n-1}(x)\overline{D}_k(x)\right)
		=D_{M_n}(x)-\psi_{M_n-1}(x)\overline{F}_{M_n}(x).
	\end{eqnarray*}
The proof is complete.
\end{proof}

\begin{corollary}\label{lemma0nnT12}
	Let $\{q_k:k\in\mathbb{N}\}$ be a sequence of non-decreasing numbers. Then for some constant $c,$
	\begin{eqnarray} \label{1.71alphaT2}
	&&\int_{G_m} F^{-1}_{M_n}(x) d\mu (x)=1, \\
	&&\sup_{n\in\mathbb{N}}\int_{G_m}\left\vert F^{-1}_{M_n}(x)\right\vert d\mu(x)\leq c<\infty,\label{1.72alphaT2} \\
	&&\sup_{n\in\mathbb{N}}\int_{G_m \backslash I_N}\left\vert F^{-1}_{M_n}(x)\right\vert d\mu (x)\rightarrow  0, \ \ \text{as} \ \ n\rightarrow  \infty, \ \ \text{for any} \ \ N\in \mathbb{N_+}.\label{1.73alphaT2}
	\end{eqnarray}
\end{corollary}
\begin{proof}
The proof is direct consequence of Proposition \ref{corollary3n} and Lemma \ref{lemma0nnT121}.
\end{proof}

\section{proof of main result}

\begin{theorem}\label{Corollary3nnconv} Let $p\geq 1$ and $\{q_k:k\in \mathbb{N}\}$ be a sequence of non-increasing numbers. Then
	$$\Vert T_n f-f\Vert_p \to 0 \ \ \text{as}\ \ n\to \infty$$
	for all  $f\in L_p(G_m)$. 
	
	Let function $f\in L_1(G_m)$ is continuous at a point $x.$ Then
	$${{T}_{n}}f(x)\to f(x), \ \ \ \text{as}\ \  \ n\to\infty.$$  Moreover,
	\begin{equation*}
	\underset{n\rightarrow \infty }{\lim }T_nf(x)=f(x)
	\end{equation*}
	for all Vilenkin-Lebesgue points of $f\in L_p(G_m)$.
	
	Let $p\geq 1$ and $\{q_k:k\in \mathbb{N}\}$ be a sequence of non-decreasing numbers satisfying the condition \eqref{fn01T1}. Then
	$$\Vert T_n f-f\Vert_p \to 0 \ \ \text{as}\ \ n\to \infty$$
	for all  $f\in L_p(G_m)$. 
	
	Let function $f\in L_1(G_m)$ is continuous at a point $x.$ Then
	$${{T}_{n}}f(x)\to f(x), \ \ \ \text{as}\ \  \ n\to\infty.$$
	Moreover,
	\begin{equation*}
	\underset{n\rightarrow \infty }{\lim }T_nf(x)=f(x)
	\end{equation*}
	for all Vilenkin-Lebesgue points of $f\in L_p(G_m)$.
\end{theorem}

\begin{proof} Let $\{q_k:k\in \mathbb{N}\}$ be a non-increasing sequence.
	Lemma \ref{lemma0nnT11} immediately follows stated norm and pointwise convergences. 
	
	Suppose that $x$ is either point of continuity or Vilenkin-Lebesgue point of function $f\in L_p(G_m).$ Then
	$$
	\underset{n\rightarrow \infty }{\lim }\vert\sigma_nf(x)-f(x)\vert=0.
	$$
	Hence,
\begin{eqnarray*}
&&\vert T_nf(x)-f(x)\vert \\
&\leq&\frac{1}{Q_n}\left(\overset{n-2}{\underset{j=0}{\sum}}\left(q_j-q_{j+1}\right)j\vert\sigma_jf(x)-f(x)\vert+q_{n-1}(n-1)\vert\sigma_{n-1}f(x)-f(x)\vert\right)
\\
&\leq&\frac{1}{Q_n}\overset{n-2}{\underset{j=0}{\sum}}\left(q_j-q_{j+1}\right)j\alpha_j+\frac{q_{n-1}(n-1)\alpha_{n-1}}{Q_n}  :=I+II, \  \text{	where  } \  \alpha_n\to 0, \  \text{as} \  n\to\infty.
\end{eqnarray*}
To prove $I\to 0, \ \text{as} \ n\to\infty$ and $II\to 0, \ \text{as} \ n\to\infty,$ we just have to analogous steps of Lemma \ref{lemma0nnT11}. It follows that Part a) is proved.
	
Now we assume that the sequence is non-decreasing and satisfying condition \eqref{fn01T1}. According to \eqref{1.72alphaT1A} in Lemma \ref{lemma0nnT1A} get norm and and pointwise convergence.
To prove  convergence in Vilenkin-Lebesgue points we use estimation 
\begin{eqnarray*}
	\vert T_nf(x)-f(x)\vert 
	&\leq&\frac{1}{Q_n}\overset{n-2}{\underset{j=0}{\sum}}\left(q_{j+1}-q_{j}\right)j\alpha_j+\frac{q_{n-1}(n-1)\alpha_n}{Q_n} \\
	&:=&III+IV, \  \text{	where  } \  \alpha_n\to 0, \  \text{as} \  n\to\infty.
\end{eqnarray*}	
		
It is evident that
	$$IV\leq\frac{q_{n-1}(n-1)\alpha_n}{Q_n}\leq \alpha_n\to 0, \ \ \text{as} \ \ n\to\infty.$$	

	On the other hand, for any $\varepsilon>0$ there exists $N_0\in \mathbb{N},$ such that $\alpha_n< \varepsilon/2$ when $n>N_0.$ We can write that
	
	\begin{eqnarray*}
		&&\frac{1}{Q_n}\overset{n-2}{\underset{j=1}{\sum}}\left(q_{j+1}-q_{j}\right)j\alpha_j 
		=\frac{1}{Q_n}\overset{N_0}{\underset{j=1}{\sum}}\left(q_{j+1}-q_{j}\right)j\alpha_j
		+\frac{1}{Q_n}\overset{n-2}{\underset{j=N_0+1}{\sum}}\left(q_{j+1}-q_{j}\right)j\alpha_j=III_1+III_2.
	\end{eqnarray*}
	Since the sequence $\{q_k\}$ is non-decreasing, we obtain that 
	$\vert q_{j+1}-q_{j}\vert<2q_{j+1}<2q_{n-1}.$   Hence,
	\begin{eqnarray*}
		III_1\leq\frac{2q_{0}N_0}{Q_n}\to 0, \ \ \ \text{as} \ \ \ n\to \infty
	\end{eqnarray*}
	and
	\begin{eqnarray*}
		III_2&\leq&\frac{1}{Q_n}\overset{n-2}{\underset{j=N_0+1}{\sum}}\left(q_{n-j-1}-q_{n-j}\right)j\alpha_j\leq\frac{\varepsilon(n-1)}{Q_n}\overset{n-2}{\underset{j=N_0+1}{\sum}}\left(q_{n-j}-q_{n-j-1}\right) \\
		&\leq& \frac{\varepsilon(n-1)}{Q_n}\left(q_{0}-q_{n-N_0}\right)\leq \frac{2q_0\varepsilon(n-1)}{Q_n}<\varepsilon.
	\end{eqnarray*}
	Therefore, also $III\to \infty$ so that the proof of part b)  is also complete.
\end{proof}

\begin{corollary}
	Let $f\in L_p,$ where  $p\geq 1.$ Then%
	\begin{eqnarray*}
		R_{n}f &\rightarrow &f,\text{ \ \ \ a.e., \ \ \ \ as \ }n\rightarrow \infty, \ \ \ \ \
		U_n^{\alpha}f \rightarrow f,\text{ \ \ \ a.e., \ \ \ as \ }n\rightarrow
		\infty ,\text{\ \ \ }\\
		V_{n}^{\alpha}f &\rightarrow &f,\text{ \ \ \ a.e., \ \ \ as \ }n\rightarrow
		\infty ,\ \ \ \ \
		B^{\alpha,\beta}_{n}f \rightarrow f,\text{ \ \ \ a.e., \ \ \ as \ }n\rightarrow
		\infty. 
	\end{eqnarray*}
\end{corollary}

\begin{theorem}\label{Corollaryconv5} Let $p\geq 1$ and $\{q_k:k\in \mathbb{N}\}$ be a sequence of non-decreasing numbers. Then
	$$\Vert T_{M_n} f-f\Vert_p \to 0 \ \ \text{as}\ \ n\to \infty$$
	for all  $f\in L_p(G_m)$. 
	
	Let function $f\in L_1(G_m)$ is continuous at a point $x.$ Then
	$${{T}_{M_n}}f(x)\to f(x), \ \ \ \text{as}\ \  \ n\to\infty.$$
	Moreover,
	\begin{equation*}
\underset{n\rightarrow \infty }{\lim }T_{M_n}f(x)=f(x), \ \ \text{for all Lebesgue points of} \ \  f\in L_p(G_m).
	\end{equation*}
\end{theorem}

\begin{proof}
Corollary \ref{lemma0nnT12} immediately follows norm and and pointwise convergence. To prove a.e convergence we use first identity in Lemma \ref{lemma0nnT121}  to write
	\begin{eqnarray*}
		T_{M_n}f\left(x\right)&=&\underset{G_m}{\int}f\left(t\right)F^{-1}_n\left(x-t\right) d\mu\left(t\right)\\
		&=&\underset{G_m}{\int}f\left(t\right)D_{M_n}\left(x-t\right)d\mu\left(t\right)
		-\underset{G_m}{\int}f\left(t\right)\psi_{M_n-1}(x-t)\overline{F}_{M_n}(x-t)=I-II.
	\end{eqnarray*}
	By applying \eqref{smnvl} we can conclude that
	$I=S_{M_n}f(x)\to f(x)$
	for all Lebesgue points of $f\in L_p(G_m)$. By using $\psi_{M_n-1}(x-t)=\psi_{M_n-1}(x)\overline{\psi}_{M_n-1}(t)$ we can conclude that
	$$II=\psi_{M_n-1}(x)\underset{G_m}{\int}f\left(t\right)\overline{F}_{M_n}(x-t)\overline{\psi}_{M_n-1}(t)d(t)$$
	By combining \eqref{concond} and Proposition \ref{corollary3n} we find that function 
	$$f\left(t\right)\overline{F}_{M_n}(x-t)\in L_p \ \ \text{ where} \ \  p\geq  1 \ \ \text{for any } \ \ x\in G_m, $$
	and $II$ is Fourier coefficients of integrable function. According to Riemann-Lebesgue Lemma  we get that
	$II\to 0 \ \ \text{for any } \ \ x\in G_m.$
	The proof is complete.
\end{proof}

\end{document}